\documentclass[10pt]{amsart}
\usepackage{appendix}
\usepackage{graphicx,enumitem,dsfont}
\usepackage{verbatim}
\usepackage[usenames]{color}
\usepackage[all]{xy}
\usepackage{amsmath, mathrsfs, amssymb, amsthm}
\usepackage[colorlinks]{hyperref}
\newtheorem{theorem}{Theorem}[section]

\newtheorem{lemma}[theorem]{Lemma}
\newtheorem{corollary}[theorem]{Corollary}
\newtheorem{remark}[theorem]{Remark}
\newtheorem{conjecture}[theorem]{Conjecture}

\newtheorem{definition}[theorem]{Definition}

\theoremstyle{definition}

\allowdisplaybreaks
\numberwithin{equation}{section}

\title[On a conjecture of Morel]{On a conjecture of Morel}

\author{Utsav Choudhury}
\address{Institut f\"ur Mathematik\\ Universit\"at Z\"urich\\ Winterthurerstrasse 190\\ CH-8057 Z\"urich\\ Switzerland}
\email{utsav.choudhury@math.uzh.ch}

\thanks{Research supported by the Swiss National Science Foundation}
\date{}
\keywords{$\mathbb{A}^1$-homotopy theory, $H$-spaces, homotopy pullback.}

\subjclass[2010]{Primary 14F42; Secondary 18F99}

\begin{document}

\begin{abstract}
In this note we prove that the $\mathbb{A}^1$-connected component sheaf $a_{Nis}(\pi_0^{\mathbb{A}^1}(\mathcal{X}))$ of an $H$-group
$\mathcal{X}$ is $\mathbb{A}^1$-invariant.   
\end{abstract}

\maketitle

\tableofcontents
\section{Introduction}

Let $Sm/k$ denote the category of smooth, separated $k$-schemes
and let \linebreak$PSh(Sm/k)$ denote the category of presheaves of sets on $Sm/k$.
A functor\linebreak$\mathcal{X} : \bigtriangleup^{op} \to PSh(Sm/k)$ is called a simplicial presheaf or a space. Here $\bigtriangleup$ is the category of simplices.
Let $\bigtriangleup^{op}PSh(Sm/k)$ denote the category of spaces.

$\bigtriangleup^{op}PSh(Sm/k)$ has a local model category structure with respect to the\linebreak Nisnevich topology called the injective Nisnevich model structure.  
A morphism $f: \mathcal{X} \to \mathcal{Y}$ is a
weak equivalence if the induced morphism on the Nisnevich stalks are weak equivalences of simplicial sets. Cofibrations are sectionwise injective\linebreak morphisms and fibrations are defined using the right lifting property (see \cite{jar1, mv}). The resulting homotopy category is denoted by $\mathbf{H}_s(Sm/k)$. 

The Bousfield localisation of the local model structure on   
\linebreak$\bigtriangleup^{op}PSh(Sm/k)$ with respect to the class of maps 
$\mathcal{X} \times \mathbb{A}^1 \to \mathcal{X}$ is called the\linebreak $\mathbb{A}^1$-model structure (the $\mathbb{A}^1$-model structure for simplicial sheaves on $Sm/k$\linebreak described in \cite{mv} was extended to simplicial presheaves in \cite{jar2}). The resulting homotopy category is denoted by $\mathbf{H}(k)$.

For any space $\mathcal{X}$, 
define $\pi_0^{\mathbb{A}^1} (\mathcal{X})$ to be the presheaf 
$$U \in Sm/k \mapsto Hom_{\mathbf{H}(k)}(U, \mathcal{X}).$$
The presheaf  $\pi_0^{\mathbb{A}^1}(\mathcal{X})$ is homotopy invariant, i.e., for any $U \in Sm/k$ the morphism $\pi_0^{\mathbb{A}^1} (\mathcal{X})(U) \to  \pi_0^{\mathbb{A}^1} (\mathcal{X})(\mathbb{A}^1_U)$, induced by the projection
$\mathbb{A}^1_U \to U$, is bijective. 
 
Let $a_{Nis} : PSh(Sm/k) \to Sh_{Nis}(Sm/k)$
denote the Nisnevich\linebreak sheafification functor. The following conjecture of Morel states that the above property remains true after Nisnevich sheafification. 

\begin{conjecture}
For any $U \in Sm/k$, the morphism $$a_{Nis}(\pi_0^{\mathbb{A}^1} (\mathcal{X}))(U) \to  a_{Nis}(\pi_0^{\mathbb{A}^1} (\mathcal{X}))(\mathbb{A}^1_U),$$ induced by the projection
$\mathbb{A}^1_U \to U$, is bijective.
\end{conjecture}

In this paper, we prove the conjecture (theorem \ref{important theorem}) for $H$-groups (definition \ref{main definition}) and homogeneous spaces on these (see definitions \ref{action definition}, \ref{homogeneous definition}). 

\medskip

\noindent
\textbf{Acknowledgements.}\,
I thank Professor  J.~Ayoub for his suggestions and guidance during this project.


\section{Generalities on the Nisnevich local model structure}
In this section we briefly recall the Nisnevich Brown-Gersten property and 
give some consequences on the $\pi_0$ functor.

Recall (\cite[Definition 3.1.3]{mv}) that a cartesian square in $Sm/k$ 
$$
 \xymatrix{                                   
W \ar[r] \ar[d] & V \ar[d]^p \\
U \ar[r]^i  & X,}
$$
is called an elementary distinguished square (in the Nisnevich topology), 
if $p$ is an \'etale morphism and $i$ is an open embedding such that $p^{-1}(X - U) \to (X-U)$
is an isomorphism (endowing these closed subsets with the reduced subscheme structure). 

A space $\mathcal{X}$ is said to satisfy the Nisnevich Brown-Gersten property if for any 
elementary distinguished square in $Sm/k$ as above, the induced square of simplicial sets
$$
 \xymatrix{                                   
\mathcal{X}(X) \ar[r] \ar[d] & \mathcal{X}(V) \ar[d] \\
\mathcal{X}(U) \ar[r]  & \mathcal{X}(W)}
 $$
is homotopy cartesian (see \cite[Definition 3.1.13]{mv}). 

Any fibrant space for the Nisnevich local model structure satisfies the Nisnevich Brown-Gersten property  (\cite[Remark 3.1.15]{mv}). 

A space is $\mathbb{A}^1$-fibrant if and only if it is fibrant in the local model structure and $\mathbb{A}^1$-local (\cite[Proposition 2.3.19]{mv}).

There exist endofunctors $Ex$ (resp. $Ex_{\mathbb{A}^1}$) of $\bigtriangleup^{op}PSh(Sm/k)$ such that for any space $\mathcal{X}$, the object $Ex(\mathcal{X})$ is fibrant
(resp. $Ex_{\mathbb{A}^1}\mathcal{X}$ is $\mathbb{A}^1$-fibrant). Moreover, there exists a natural morphism $\mathcal{X} \to Ex(\mathcal{X})$ (resp. $\mathcal{X} \to Ex_{\mathbb{A}^1}(\mathcal{X})$) which is a local weak\linebreak equivalence (resp.   $\mathbb{A}^1$-weak equivalence) (\cite[Remark 3.2.5, Lemma 3.2.6, Theorem 2.1.66]{mv}).

\begin{remark} \label{localisation}
For the injective local model structure all spaces are cofibrant. Hence for any space $\mathcal{X}$ and for any $U \in Sm/k$,    
$$Hom_{\mathbf{H}_s(Sm/k)}(U, \mathcal{X}) = \pi_0(Ex(\mathcal{X})(U)).$$ 
Since $Ex_{\mathbb{A}^1}(\mathcal{X})$ is $\mathbb{A}^1$-local, 
$$Hom_{\mathbf{H}(k)}(U, \mathcal{X}) = Hom_{\mathbf{H}_s(Sm/k)}(U, Ex_{\mathbb{A}^1}(\mathcal{X})).$$
Moreover  $Ex_{\mathbb{A}^1}(\mathcal{X})$ is fibrant. Hence,
$$Hom_{\mathbf{H}(k)}(U, \mathcal{X})=
\pi_0(Ex_{\mathbb{A}^1}(\mathcal{X})(U)).$$
\end{remark}

For any space $\mathcal{X}$, let $\pi_0(\mathcal{X})$ be the presheaf defined by 
$$U \in Sm/k \mapsto Hom_{\mathbf{H}_s(Sm/k)}(U, \mathcal{X}).$$

\begin{theorem} \label{a1 surjective theorem}
Let $\mathcal{X}$ be a space. 
For any $X \in Sm/k$, such that $dim(X) \leq 1$, the canonical morphism $$\pi_0(\mathcal{X})(X) \to a_{Nis}(\pi_0(\mathcal{X}))(X)$$
is surjective.
\end{theorem}

Before giving the proof we note the following consequence.
\begin{corollary} \label{a1 surjective}
For any space $\mathcal{X}$, the canonical morphism 
$$\pi_0^{\mathbb{A}^1}(\mathcal{X})(\mathbb{A}^1_F) \to a_{Nis}(\pi_0^{\mathbb{A}^1}(\mathcal{X}))(\mathbb{A}^1_F)$$ 
is bijective for all finitely generated separable field extensions $F/k$.
\end{corollary}

\begin{proof}
For any $X \in Sm/k$,
$$\pi_0^{\mathbb{A}^1}(\mathcal{X})(X) = \pi_0(Ex_{\mathbb{A}^1}\mathcal{X})(X).$$
The canonical morphism 
$$ \pi_0^{\mathbb{A}^1}(\mathcal{X})(\mathbb{A}^1_F) \to a_{Nis}(\pi_0^{\mathbb{A}^1}(\mathcal{X}))(\mathbb{A}^1_F) $$
is surjective (applying theorem \ref{a1 surjective theorem} for the space $Ex_{\mathbb{A}^1}(\mathcal{X})$). On the other hand, consider the following commutative diagram
$$\xymatrix{                                   
\pi_0^{\mathbb{A}^1}(\mathcal{X})(\mathbb{A}^1_F) \ar[r] \ar[d] & \pi_0^{\mathbb{A}^1}(\mathcal{X})(F) \ar[d]^{\wr} \\
a_{Nis}(\pi_0^{\mathbb{A}^1}(\mathcal{X}))(\mathbb{A}^1_F) \ar[r]  & a_{Nis}(\pi_0^{\mathbb{A}^1}(\mathcal{X}))(F),}
$$
where the horizontal morphisms are induced by the zero section $F \xrightarrow{s_0} \mathbb{A}^1_F$. The top horizontal morphism and the right vertical morphism are bijective. Hence the left vertical surjective morphism is injective.
\end{proof}

The proof of  theorem \ref{a1 surjective theorem} depends on the relation between homotopy pullback of spaces and pullback of the presheaves of connected components of those spaces. 

Let $I$ be a small category. There is a functor $(I/-) : I \to Cat$ such that for any $i \in I$, $(I/-)(i) = I/i$. Here $Cat$ is the category of small categories and $I/i$ is the over category. There is a functor $N : Cat \to \bigtriangleup^{op}Sets$, such that for any $J \in Cat$, the simplicial set $N(J)$ is the nerve of the category $J$.
Define $N(I/-) := N \circ (I/-)$. 

A set $S$ will be considered as a simplicial set in the obvious way : in every simplicial degree it is given by $S$ and faces and degeneracies are identities. These simplicial sets are called discrete simplicial sets. 

\begin{lemma}
Let $X : I \to \bigtriangleup^{op}Sets$ be a diagram of discrete simplicial sets. Then $lim_I X \cong holim_I X$.  
\end{lemma}

\begin{proof}
By adjointness (\cite[Ch. XI 3.3]{bk}) $$Hom(\bigtriangleup^n \times N(I/-), X) = Hom(\bigtriangleup^n, holim_I X).$$ The fucntor $\pi_0 : (\bigtriangleup^{op}Sets)^I \to (Sets)^I$
is left adjoint to the fucntor\linebreak $N : (Sets)^I \to (\bigtriangleup^{op}Sets)^I$, where $N$ maps a diagram of sets to the same diagram of discrete simplicial sets.
Hence $Hom(\bigtriangleup^n \times N(I/-), X) = Hom(\bullet_I , X)$, where $\bullet_I$ is the diagram of sets given by the one element set for each $i \in I$. 
But $Hom(\bullet_I , X) = Hom(\bullet, lim_I X)$, by adjointness. Therefore, we get our result.
\end{proof}

\begin{remark}
Let $X : I \to \bigtriangleup^{op}Sets$ be a diagram such that each $X(i)$ is fibrant for all $i \in I$. The canonical morphism $X(i) \to \pi_0(X(i))$ induces a morphism \linebreak $holim_I(X) \to lim_I \pi_0(X)$. This gives the following morphism 
\begin{equation} \label{homotopy limit map}
\pi_0(holim_I(X)) \to lim_I \pi_0(X).
\end{equation}
\end{remark}

\begin{lemma} \label{surjective}
 Suppose that $I$ is the pullback category $1 \rightarrow 0 \leftarrow 2$ and let \linebreak$D :I \to \bigtriangleup^{op}Sets$ be a digram $X \xrightarrow{p} Y \xleftarrow{q} Z$ 
such that $X,Y,Z$ are fibrant. Then the map \eqref{homotopy limit map} is surjective.
\end{lemma}

\begin{proof}
By \cite[Ch. XI 4.1.(iv), 5.6]{bk}  $holim_I(X) \cong X' \times_Y Z$, where $X \to X' \xrightarrow{p'} Y$ is a factorisation
of $p$ into a trivial cofibration followed by a fibration $p'$.  Since $\pi_0(X) \cong \pi_0(X')$, it is enough to show that 
$$\pi_0 (X' \times_Y Z) \to \pi_0(X') \times_{\pi_0(Y)} \pi_0(Z)$$ is surjective. So we can assume that $p$ is a fibration.
Let $s \in \pi_0(X) \times_{\pi_0(Y)} \pi_0(Z)$. $s$ can be represented (not uniquely) by
$(x,y, z)$, where $(x,z) \in X_0 \times Z_0$ and $y \in Y_1$ such that $d_0(y) = p(x)$ and $d_1(y) = q(z)$. Since $p$ is a fibration, we can lift the path $y$ to a path $y' \in X_1$
such that $d_0(y') = x$ and $x':= d_1(y')$ maps to $q(z)$.  $holim_I D \cong X \times_Y Z$. Therefore $(x',z) \in holim_I D$ which maps to $s$. This proves the surjectivity.
\end{proof}

\begin{remark}



Under the condition of lemma \ref{surjective}, the map \eqref{homotopy limit map} may not be injective. Indeed, if $Y$ is connected, $X$ is the universal cover of $Y$ and $Z = \bullet$,  then \eqref{homotopy limit map} is injective if and only if $Y$ is simply connected.


\end{remark}



A noetherian $k$-scheme $X$, which is the inverse limit of a left filtering system $(X_{\alpha})_{\alpha}$ with each transition morphism $X_{\beta} \to X_{\alpha}$ being an \'etale affine morphism between smooth $k$-schemes, is called an essentially smooth $k$-scheme. For any $X \in Sm/k$ and any $x \in X$, the local schemes 
$Spec(O_{X,x})$ and $Spec(O^h_{X,x})$ are essentially smooth $k$-schemes.

\begin{lemma} \label{a1 surjective lemma}
Let $\mathcal{X}$ be a space. For any essentially smooth discrete valuation ring $R$, the canonical morphism 
$$\pi_0(\mathcal{X})(R) \to a_{Nis}(\pi_0(\mathcal{X}))(R)$$ is surjective.
\end{lemma}

\begin{proof}
By remark \ref{localisation} we can assume that $\mathcal{X}$ is fibrant.

Let $F = Frac(R)$ and let $R^h$ be the henselisation of $R$ at its maximal ideal. Suppose $s \in a_{Nis}(\pi_0(\mathcal{X}))(R)$. Then for the image of $s$ in 
$a_{Nis}(\pi_0(\mathcal{X}))(R^h)$, there exists a Nisnevich neighbourhood of the closed point $p : W \to Spec(R)$ and \linebreak$s' \in \pi_0(\mathcal{X})(W)$, such that $s'$ gets mapped to $s|_W \in  a_{Nis}(\pi_0(\mathcal{X}))(W)$.  Let \linebreak$L = Frac(W)$. For any finitely generated separable field extension $F/k$, the map 
$\pi_0(\mathcal{X})(F) \to a_{Nis}(\pi_0(\mathcal{X}))(F)$ is bijective.
Hence, $s'|_{L}$ is same as $s|_L$. We get two sections $s' \in
 \pi_0(\mathcal{X})(W)$ and $s|_F \in \pi_0(\mathcal{X})(F)$, such that $s'|_L = s|_L$. By lemma \ref{surjective} and the fact that $\mathcal{X}$ satisfies the Nisnevich Brown-Gersten property, we find an element $s_v \in \pi_0(\mathcal{X})(R)$ which gets mapped to $s$. Therefore, $\pi_0(\mathcal{X})(R) \to a_{Nis}(\pi_0(\mathcal{X}))(R)$ is surjective. 
\end{proof}

\begin{proof} [Proof of theorem \ref{a1 surjective theorem}]
Let $X \in Sm/k$ and $dim(X) =1$. Let $\alpha$ be an element of\linebreak 
$a_{Nis}(\pi_0(\mathcal{X}))(X)$. This $\alpha$ gives $\alpha_p \in a_{Nis}(\pi_0(\mathcal{X}))(O_{X,p})$ for every codimension $1$ point $p \in X$, such that 
$\alpha_p|_{K(X)} = \alpha_q|_{K(X)}$, for all $p,q \in X^{(1)}$. By the surjectivity of
  $$\pi_0(\mathcal{X})(O_{X,p}) \to a_{Nis}(\pi_0(\mathcal{X}))(O_{X,p})$$ and bijectivity
of  $$\pi_0(\mathcal{X})(K(X)) \to a_{Nis}(\pi_0(\mathcal{X}))(K(X)),$$ we get elements
$\alpha'_p \in  \pi_0(\mathcal{X})(O_{X,p})$ mapping to $\alpha_p$, such that
$\alpha'_p|_{K(X)} = \alpha'_q|_{K(X)}$ for $p,q \in X^{(1)}$. 

Fix a $p \in X^{(1)}$. There exists an open set $U$ and $\beta \in \pi_0(\mathcal{X})(U)$, such that $\beta|_{O_{X,p}} = \alpha'_p$. Let $\beta' \in a_{Nis}(\pi_0(\mathcal{X})(U))$ be the image of $\beta$. Suppose that $\beta' \neq \alpha|_U$, but $\beta'|_{O_{X,p}} = \alpha_p$. Hence there exsists $U' \subset U$, such that $\beta'|_{U'} =
\alpha|_{U'}$.

So we can assume by Noetherian property of $X$ that there exists a maximal open set
$U \subset X$ and $\alpha' \in  \pi_0(\mathcal{X})(U)$, such that $\alpha'$ gets mapped to $\alpha|_U$. If $U \neq X$, then there exists a codimension one point $q \in X \setminus U$. We can get an open neighborhood $U_q$ and an element $\alpha'' \in \pi_0(\mathcal{X})(U_q)$, such that $\alpha''$ gets mapped to $\alpha|_{U_q}$. But by construction of these $\alpha'', \alpha'$ we know that $\alpha''|_{K(X)} = \alpha'|_{K(X)}$. Hence there exists an open set $U' \subset U_q \cap U$, such that $\alpha''|_{U'} = \alpha'|_{U'}$. Let $Z = U_q \cap U \setminus U'$. Since $dim(X) = 1$, the set 
$Z$ is finite collection of closed points. Therefore, $Z$ is closed in $U$. Let $U'' = U \setminus Z$ be the open subset of $U$. Note that $U'' \cap U_q = U'$. Denote $U'' \cup U_q = U \cup U_q $ by $V$.

Let $\alpha'|_{U''} \in \pi_0(\mathcal{X})(U'')$ be the restriction of $\alpha'$ to $U''$. Hence, $\alpha'|_{U''}$ gets mapped to $\alpha|_{U''}$ and $\alpha'|_{U''}$ restricted to $U'$ is same as $\alpha''$ restricted to $U'$. As $\mathcal{X}$ is Nisnevich fibrant, it satisfies 
the Zariski Brown-Gersten property. By lemma \ref{surjective}, we get a section $s_V \in \pi_0(\mathcal{X})(V)$ which gets mapped to $s|_{V}$. This gives a contradiction to the maximality of $U$. This finishes the proof of the theorem.
\end{proof}

\section{$H$-groups and homogeneous spaces}
In this section we prove $\mathbb{A}^1$-invariance of $a_{Nis}(\pi_0^{\mathbb{A}^1})$ for $H$-groups and\linebreak homogeneous spaces for
$H$-groups. 

\begin{definition} \label{main definition}
Let $\mathcal{X}$ be a pointed space, i.e., $\mathcal{X}$ is a space endowed with a morphism $x : Spec(k) \to \mathcal{X}$. It is called an $H$-space if there exists a base point preserving morphism
$\mu : (\mathcal{X} \times \mathcal{X}) \to \mathcal{X}$, such that
$\mu \circ (x \times id_{\mathcal{X}})$ and $\mu \circ (id_{\mathcal{X}} \times x )$ are equal to $id_{\mathcal{X}}$ in $\mathbf{H}(k)$.  Here $\mathcal{X} \times \mathcal{X}$ is pointed by $(x,x)$.
It is called an $H$-group if :
\begin{enumerate}
 \item $\mu \circ (\mu \times id_{\mathcal{X}})$ is equal to $\mu \circ (id_{\mathcal{X}} \times \mu)$ in $\mathbf{H}(k)$ modulo the canonical isomorphism $\alpha : \mathcal{X} \times (\mathcal{X} \times \mathcal{X})
\to (\mathcal{X} \times \mathcal{X}) \times \mathcal{X}.$ 

\item There exists a morphism $(.)^* : \mathcal{X} \to \mathcal{X}$, such that $\mu \circ (id_{\mathcal{X}}, (.)^*)$ and $\mu \circ ((.)^* , id_{\mathcal{X}})$ are equal to the constant map $c : \mathcal{X} \to \mathcal{X}$ in $\mathbf{H}(k)$. Here the image of the constant map $c$ is $x$.
\end{enumerate}
\end{definition}

\begin{remark} \label{associative}
Recall from \cite[§3.2.1]{mv} that 
$$Ex_{\mathbb{A}^1} = Ex^{\mathcal{G}} \circ (Ex^{\mathcal{G}} \circ Sing_*^{\mathbb{A}^1})^{\mathbb{N}} \circ Ex^{\mathcal{G}}.$$
The fucntors $Ex^{\mathcal{G}}$ and $Sing_*^{\mathbb{A}^1}$ commutes with finite limits by \cite[§2.3.2, Theorem 2.1.66]{mv}. Also filtered colimit commutes with finite products. Therefore, $Ex_{\mathbb{A}^1}$ commutes with finite products. If $\mathcal{X}$ is an $H$-group as described in \ref{main definition}, then 
the morphisms $Ex_{\mathbb{A}^1}(x)$, $Ex_{\mathbb{A}^1}(\mu)$ and $Ex_{\mathbb{A}^1}((.)^*)$ satisfy the conditions of the definition \ref{main definition}.
Hence, $Ex_{\mathbb{A}^1}(\mathcal{X})$ is also an $H$-group.

Suppose that $a,b,c \in \pi_0(Ex_{\mathbb{A}^1}(\mathcal{X}))(U)$
for some $U \in Sm/k$. Let $f,g : \mathcal{Y} \to \mathcal{Z}$ be morphisms between $\mathbb{A}^1$-fibrant spaces such that $f$ is equal to $g$ in $\mathbf{H}(k)$, then $f$ and $g$ are simplicially homotopic. Using this, we get $\mu(a, \mu(b,c)) = \mu(\mu(a,b),c)$, $\mu(a, x) = a = \mu(x,a)$ and $\mu(a,a^*) = \mu(a^*,a) = x$. 
Hence, $\pi_0(Ex_{\mathbb{A}^1}(\mathcal{X}))$ is a presheaf of groups.
\end{remark}

Let $\mathcal{X}$ be an $H$-group. Let $\mathcal{Y}$ be a space. 
\begin{definition} \label{action definition}
The space $\mathcal{Y}$ is called an $\mathcal{X}$-space if 
there exists a morphism\linebreak
$a : \mathcal{X} \times \mathcal{Y} \to \mathcal{Y}$,
such that the following diagram
 $$\xymatrix{
\mathcal{X} \times (\mathcal{X} \times \mathcal{Y})  \ar[r]^-{id_{\mathcal{X}} \times a} \ar[d]^{a_{\mathcal{X}} \times id_{\mathcal{Y}}} & \mathcal{X} \times \mathcal{Y} \ar[d]^a \\
\mathcal{X} \times \mathcal{Y} \ar[r]^a  & \mathcal{Y}}$$
commutes in $\mathbf{H}(k)$.
\end{definition}

\begin{definition} \label{homogeneous definition}
Let $\mathcal{X}$ be an $H$-group and let $\mathcal{Y}$ be an $\mathcal{X}$-space.  $\mathcal{Y}$ is called a homogeneous $\mathcal{X}$-space if for any essentially smooth henselian $R$, the presheaf of groups $\pi_0^{\mathbb{A}^1}(\mathcal{X})(R)$ acts transitively on $\pi_0^{\mathbb{A}^1}(\mathcal{Y})(R)$.
\end{definition}

\begin{remark} \label{trivial}
$\mathcal{Y}$ is a homogeneous $\mathcal{X}$-space if and only if the colimit of
the diagram $\pi_0^{\mathbb{A}^1}(\mathcal{Y}) \xleftarrow{pr} \pi_0^{\mathbb{A}^1}(\mathcal{X}) \times \pi_0^{\mathbb{A}^1}(\mathcal{Y}) \xrightarrow{a} \pi_0^{\mathbb{A}^1}(\mathcal{Y})$ is Nisnevich locally trivial.
\end{remark}

\begin{lemma} \label{commute}
If  
$$\xymatrix{
B  \ar[r]\ar[d] &C\ar[d] \\
A \ar[r]  & D}$$
is a homotopy cocartesian square of spaces then, after applying $a_{Nis}(\pi_0)$, one gets a cocartesian square of sheaves.
\end{lemma}

\begin{proof}
Let $S \in PSh(Sm/k)$ and let $\iota(S)$ be the simplicial presheaf such that in every 
simplicial degree $k$, $\iota(S)_k = S$. The face and degeneracy morphisms are identity morphisms. This gives a functor  $\iota : PSh(Sm/k) \to \bigtriangleup^{op}PSh(T)$ which is right adjoint to $\pi_0$.
Hence $a_{Nis}(\pi_0)$ also has a right adjoint $\iota : Sh(Sm/k)
\to  \bigtriangleup^{op}Sh(Sm/k)$. This implies, $a_{Nis}(\pi_0)$ commutes with colimits. Let $B \xrightarrow{f} A' \xrightarrow{g} A$ be a factorisation of $B \to A$, such that $f$ is a cofibration and $g$ is a trivial fibration. Homotopy colimit of the digram $A \xleftarrow{} B \rightarrow{} C$ is weakly equivalent to the colimit of $A' \xleftarrow{} B \rightarrow{} C$. As $a_{Nis}(\pi_0)$ commutes with colimits and $a_{Nis}(\pi_0(A)) \cong a_{Nis}(\pi_0(A'))$, we get our result.
\end{proof}

\begin{corollary} \label{transitive}
Let $\mathcal{Y}$ be an $\mathcal{X}$-space. $\mathcal{Y}$ is a homogeneous $\mathcal{X}$-space if and only if the homotopy pushout of $Ex_{\mathbb{A}^1}(\mathcal{Y}) \xleftarrow{pr} Ex_{\mathbb{A}^1}(\mathcal{X}) \times Ex_{\mathbb{A}^1}(\mathcal{Y}) \xrightarrow{a} Ex_{\mathbb{A}^1}(\mathcal{Y})$ is connected.
\end{corollary}

\begin{proof}
The proof follows from lemma \ref{commute} and remark \ref{trivial}.
\end{proof}



\begin{lemma} \label{Nisnevich locally transitive}
Let $\mathcal{Y}$ be an $\mathcal{X}$-space. $\mathcal{Y}$ is a homogeneous $\mathcal{X}$-space if the homotopy pushout of
$\mathcal{Y} \xleftarrow{pr} \mathcal{X} \times \mathcal{Y} \xrightarrow{a} \mathcal{Y}$ is connected.
\end{lemma}
\begin{proof}
By \cite[corolarry 2.3.22]{mv}, the canonical morphism\linebreak
$a_{Nis}(\pi_0(\mathcal{X})) \to a_{Nis}(\pi_0^{\mathbb{A}^1}(\mathcal{X}))$ (resp.  
$a_{Nis}(\pi_0(\mathcal{Y})) \to a_{Nis}(\pi_0^{\mathbb{A}^1}(\mathcal{Y}))$ is surjective as morphism of Nisnevich sheaves. Hence, Nisnevich locally the action of $a_{Nis}(\pi_0^{\mathbb{A}^1}(\mathcal{X}))$ on 
$a_{Nis}(\pi_0^{\mathbb{A}^1}(\mathcal{Y}))$ is transitive.
\end{proof}





\begin{lemma} \label{pointed injective}
Let $G, G'$ be groups acting on pointed sets $S, S'$ by action maps $a, a'$ respectively. Suppose that $f : G \to G'$ is a group homorphism and let $s : S \to S'$ be a morphism of pointed sets with trivial kernel such that $s \circ a = a' \circ (f \times s)$. If 
$G$ acts transitively on $S$, then $s$ is injective.
\end{lemma}
\begin{proof}
Let $b_S$ (resp. $b_{S'}$) be the base point of $S$ (resp. $S'$) and let $a,b \in S$. Since $G$ acts transitively on $S$, there exist $g,g' \in G$ such that
$a(g, b_S) = a$ and $a(g', b_S) = b$. If $s(a) = s(b)$, then $a'(f(g), b_{S'}) = a'(f(g'), b_{S'})$. Hence
$a'(f(g^{-1}.g'), b_{S'}) =  b_{S'}$. So $s(a(g^{-1}.g', b_S)) = b_{S'}$. But $s$ is a morphism of pointed sets with trivial kernel, therefore $a(g^{-1}.g', b_S) = b_S$. This implies $a = a(g, b_S) = a(g', b_S) = b$.
\end{proof}

Let $\tilde{Sm/k}$ be the category whose objects are same as objects of $Sm/k$, but the morphisms are smooth morphisms. The following argument is taken from 
\cite[Corollary 5.9]{mor}

\begin{lemma} \label{Nisnevich injective}
Let $S$ be a Nisnevich sheaf on $Sm/k$. Suppose that for all essentially smooth henselian $X$,  the map $S(X) \to S(K(X))$ is injective. Then $S(Y) \to S(K(Y))$ is injective, for all connected $Y \in Sm/k$.
\end{lemma}

\begin{proof}
Let $S'$ be the presheaf on $\tilde{Sm/k}$, given by $$X \in \tilde{Sm/k}\mapsto \prod_i S(K(X_i)),$$ where $X_i$'s are the connected components of $X$. Then $S'$ is a
Nisnevich sheaf on $\tilde{Sm/k}$ (as every Nisnevich covering of some $X \in \tilde{Sm/k}$ splits over some open dense $U \subset X$). 
The canonical morphism $S \to S'$ is injective on Nisnevich stalks. Hence $S \to S'$ is sectionwise injective.
\end{proof}

\begin{corollary} \label{Zariski injective}
Let $S$ be a Nisnevich sheaf on $Sm/k$. Suppose that for all essentially smooth henselian $X$, the map  $S(X) \to S(K(X))$ is injective. Then $S(Y) \to S(U)$ is injective for any $Y \in Sm/k$ and any open dense $U \subset Y$.
\end{corollary}

\begin{proof}
We can assume that $Y$ is connected. By lemma \ref{Nisnevich injective}, the morphism 
$S(Y) \to S(K(Y))$ is injective and $S(U) \to S(K(Y))$ is injective, hence $S(Y) \to S(U)$ is injective. 
\end{proof}



\begin{lemma} \label{A1 invariance}
Let $S$ be a Zariski sheaf on $Sm/k$, such that 
$S(X) \to S(U)$ is injective for any $X \in Sm/k$ and for any open dense $U \subset X$.
Then $S$ is $\mathbb{A}^1$-invariant if and only if $S(F) \to S(\mathbb{A}^1_F)$ is bijective for every finitely generated separable field extension $F/k$.
\end{lemma}

\begin{proof}
The only if part is clear. We need to show that for any connected $X \in Sm/k$, the morphism $S(\mathbb{A}^1_X) \to S(X)$ (induced by the zero section), is bijective. Let $F := K(X)$. 
In the following commutative diagram
$$\xymatrix{                                   
S(\mathbb{A}^1_X) \ar[r] \ar[d] & S(X) \ar[d] \\
S(\mathbb{A}^1_F) \ar[r]  & S(F)}$$
the left vertical, the right vertical and the bottom horizontal morphisms are injective, thus the top horizontal surjective morphism is injective.
\end{proof}

We recall the following from \cite{thk} and \cite[Corollary 5.7]{mor} 

\begin{theorem} \label{gabbers lemma}
Let $X$ be a smooth (or essentially smooth) $k$-scheme, $s \in X$ be a point and $Z \subset X$ be a closed subscheme of codimension $d >0$. Then there exists an open subscheme $\Omega \subset X$ containing $s$ and a closed subscheme $Z' \subset \Omega$, of codimension $d-1$, containing $Z_{\Omega} := Z \cap \Omega$ and such that for any $n \in \mathbb{N}$ and for any $\mathbb{A}^1$-fibrant space $\mathcal{X}$,
the map 
$$\pi_n(\mathcal{X}(\Omega/(\Omega - Z_{\Omega}))) \to \pi_n(\mathcal{X}(\Omega/(\Omega-Z')))$$
is the trivial map.
In particular, if $Z$ has codimension $1$ and $X$ is irreducible,
$Z'$ must be $\Omega$. Thus for any $n \in \mathbb{N}$ the map
$$\pi_n(\mathcal{X}(\Omega/(\Omega - Z_{\Omega}))) \to \pi_n(\mathcal{X}(\Omega))$$ is the trivial map. 
\end{theorem}

\begin{remark}[\cite{mor}] \label{gabbers lemma application}
Let $X$ be an essentially smooth local ring 
and let $x$ be the closed point. Let $U \subset X$ be an open set. We have the following exact sequence of sets and groups for any $\mathbb{A}^1$-fibrant space $\mathcal{X}$
  :
$$\dots \to \pi_1(\mathcal{X})(X) \to \pi_1(\mathcal{X})(U)
\to  \pi_0(\mathcal{X})(X/U) \to \pi_0(\mathcal{X})(X) \to \pi_0(\mathcal{X})(U)$$
Applying theorem \ref{gabbers lemma} to $X$ and its closed point $x$, we see that $\Omega = X$ and the 
morphisms
$$\pi_n(\mathcal{X})(X/U) \to \pi_n(\mathcal{X})(X)$$
are trivial. Hence the morphism of pointed sets 
$$\pi_0(\mathcal{X})(X) \to \pi_0(\mathcal{X})(U)$$
has trivial kernel. Taking colimit over open sets, this gives
the morphism of pointed sets
$$\pi_0(\mathcal{X})(X) \to \pi_0(\mathcal{X})(K(X))$$ 
which has trivial kernel. In particular if $X$ is henselian, then the morphism of pointed sets
$$a_{Nis}(\pi_0(\mathcal{X}))(X) \to a_{Nis}(\pi_0(\mathcal{X}))(K(X))$$
has trivial kernel.
\end{remark}

\begin{theorem} \label{important theorem}
Let $\mathcal{X}$ be an $H$-group and $\mathcal{Y}$ be a homogeneous $\mathcal{X}$-space. Then $a_{Nis}(\pi_0^{\mathbb{A}^1}(\mathcal{X}))$ and $a_{Nis}(\pi_0^{\mathbb{A}^1}(\mathcal{Y}))$ are $\mathbb{A}^1$-invariant.
\end{theorem}

\begin{proof}
For any connected $X \in Sm/k$ and any $x \in X$, the morphisms of pointed sets
$$a_{Nis}(\pi_0^{\mathbb{A}^1}(\mathcal{X}))(O_{X,x}^h) \to a_{Nis}(\pi_0^{\mathbb{A}^1}(\mathcal{X}))(K(O_{X,x}^h))$$
$$a_{Nis}(\pi_0^{\mathbb{A}^1}(\mathcal{Y}))(O_{X,x}^h) \to a_{Nis}(\pi_0^{\mathbb{A}^1}(\mathcal{Y}))(K(O_{X,x}^h))$$
have trivial kernel by remark \ref{gabbers lemma application}. By
lemma \ref{pointed injective} and the fact that $a_{Nis}(\pi_0^{\mathbb{A}^1}(\mathcal{X}))(O_{X,x}^h)$ is a group, the morphisms mentioned above are injective morphisms of sets.  By lemma \ref{Nisnevich injective}, for every $X \in Sm/k$, the morphisms 
$$a_{Nis}(\pi_0^{\mathbb{A}^1}(\mathcal{X}))(X) \to a_{Nis}(\pi_0^{\mathbb{A}^1}(\mathcal{X}))(K(X))$$
and 
$$a_{Nis}(\pi_0^{\mathbb{A}^1}(\mathcal{Y}))(X) \to a_{Nis}(\pi_0^{\mathbb{A}^1}(\mathcal{Y}))(K(X))$$
are injective. Hence for any $X \in Sm/k$ and any open dense subscheme $U \subset X$, the morphisms
$$a_{Nis}(\pi_0^{\mathbb{A}^1}(\mathcal{X}))(X) \to a_{Nis}(\pi_0^{\mathbb{A}^1}(\mathcal{X}))(U)$$
and 
$$a_{Nis}(\pi_0^{\mathbb{A}^1}(\mathcal{Y}))(X) \to a_{Nis}(\pi_0^{\mathbb{A}^1}(\mathcal{Y}))(U)$$
are injective  by corollary \ref{Zariski injective},.
Now applying corollary \ref{a1 surjective} and lemma \ref{A1 invariance}, we get our result.
\end{proof}

\begin{remark}
If $\mathcal{X}$ is an $H$-group, then $$\pi_0^{\mathbb{A}^1}(\mathcal{X})(R) \to a_{Nis}(\pi_0^{\mathbb{A}^1}(\mathcal{X}))(R)$$ 
is bijective for any essentially smooth discrete valuation ring $R$. Indeed, using remark \ref{gabbers lemma application} one can easily show that for any essentially smooth discrete valuation ring $R$, the group homomorphism $$\pi_0^{\mathbb{A}^1}(\mathcal{X})(R) \to \pi_0^{\mathbb{A}^1}(\mathcal{X}))(K(R))$$ is injective. On the other hand, consider the following commutative diagram
$$\xymatrix{
\pi_0^{\mathbb{A}^1}(\mathcal{X})(R)  \ar[r]\ar[d] &\pi_0^{\mathbb{A}^1}(\mathcal{X}))(K(R))\ar[d]^{\wr} \\
 a_{Nis}(\pi_0^{\mathbb{A}^1}(\mathcal{X}))(R)\ar[r]  & a_{Nis}(\pi_0^{\mathbb{A}^1}(\mathcal{X}))(K(R)) }$$
where the bottom horizontal morphism is injective by theorem \ref{important theorem}.
The right vertical injective morphism is surjective by lemma \ref{a1 surjective lemma}. Hence it is bijective.  
\end{remark}

\section{Application and comments}
By gathering known facts from \cite{Wen}, \cite[Theorem 2.4]{mor1} and \cite[Corollary 5.10]{gi} one can show that for any connected linear algebraic group $G$, such that the almost simple factors of the universal covering (in algebraic group theory sense) of the semisimple part of $G$ is isotropic and retract $k$-rational (\cite[Definition 2.2]{gi}), the sheaf $a_{Nis}(\pi_0^{\mathbb{A}^1}(G))$ is
$\mathbb{A}^1$-invariant. By \ref{important theorem}, we have the following generalisation.

\begin{corollary}
Let $G$ be any sheaf of groups on $Sm/k$ and $B$ be any subsheaf of groups. Then 
$a_{Nis}(\pi_0^{\mathbb{A}^1}(G))$ is $\mathbb{A}^1$-invariant and $a_{Nis}(\pi_0^{\mathbb{A}^1}(G/B))$ is $\mathbb{A}^1$-invariant. Here $G/B$ is the quotient sheaf in Nisnevich topology.
\end{corollary}

We recall from \cite[Definition 7]{mor} the following definition.
\begin{definition}
A sheaf of groups $G$ on $Sm/k$ is called strongly $\mathbb{A}^1$-invariant if for any $X \in Sm/k$, the map
$$H^i_{Nis}(X, G) \to H^i_{Nis}(\mathbb{A}^1_X, G)$$
induced by the projection $\mathbb{A}^1_X \to X$, is bijective for $i \in \left\{0,1\right\}$.
\end{definition}

Let $\mathcal{X}$ be a pointed space.
By \cite[Theorem 9]{mor}, for any pointed simplicial persheaf $\mathcal{X}$, the sheaf of groups $a_{Nis}(\pi_0(\Omega(Ex_{\mathbb{A}^1}(\mathcal{X}))))= \pi_1^{\mathbb{A}^1}(\mathcal{X},x)$ is strongly 
$\mathbb{A}^1$-invariant. Here $x$ is the base point of $\mathcal{X}$ and $\Omega(Ex_{\mathbb{A}^1}(\mathcal{X}))$ is the loop space of $Ex_{\mathbb{A}^1}(\mathcal{X})$. So for any space $\mathcal{X}$, which is the loop space of some $\mathbb{A}^1$-local space $\mathcal{Y}$, \cite[Thoerem 9]{mor} gives the $\mathbb{A}^1$-invariance property for $a_{Nis}(\pi_0^{\mathbb{A}^1}(\mathcal{X}))$.
We end this section by showing that there exists an $\mathbb{A}^1$-local $H$-group which is not a loop space of some $\mathbb{A}^1$-local space. This will imply that the statement of the theorem \ref{important theorem} for $H$-groups is not a direct consequence of \cite[Theorem 9]{mor}. It is enough to show that there exists sheaf of 
groups $G$ which is $\mathbb{A}^1$-invariant, but not strongly $\mathbb{A}^1$-invariant.
\bigskip

Let $\mathbb{Z}[\mathbb{G}_m]$ be the free presheaf of abelian groups generated by $\mathbb{G}_m$. 

\begin{remark} \label{remark injective open}
For any $X \in Sm/k$ and a dominant morphism $U \to X$, the canonical morphism $ \mathbb{Z}[\mathbb{G}_m](X) \to \mathbb{Z}[\mathbb{G}_m](U)$ is injective. Indeed, any nonzero $a \in \mathbb{Z}[\mathbb{G}_m](X)$ can be written as $a = \sum_{i=1}^n a_i.g_i$, where $g_i \in \mathbb{G}_m(X)$ and $a_i \in \mathbb{Z} \setminus \left\{0\right\}$ such that $g_i \neq g_{i'}$ for $i \neq i'$. Suppose $a|_U = 0$, i.e.,  $\sum_{i=1}^n a_i.g_i|_U = 0$. Since $\mathbb{G}_m(X) \to \mathbb{G}_m(U)$ is injective, $g_i|_U \neq g_{i'}|_U$ for $i \neq i'$. This implies $a_i = 0$ for all $i$. Hence $a = 0$. 
\end{remark}

The presheaf $\mathbb{Z}[\mathbb{G}_m]$ is not a Nisnevich sheaf.  
But it is not far from being a Nisnevich sheaf. 

\begin{lemma}
The Nisnevich sheafification $a_{Nis}(\mathbb{Z}[\mathbb{G}_m])$ is the presheaf that associates to every smooth $k$-scheme $X = \coprod_i X_i$, the abelian group $\prod_i  \mathbb{Z}[\mathbb{G}_m](X_i)$, where $X_i$'s are the connected components of $X$.
\end{lemma}

\begin{proof}
Let $\mathcal{F}$ be the presheaf that associates to every smooth $k$-scheme $X = \coprod_i X_i$, the abelian group $\prod_i  \mathbb{Z}[\mathbb{G}_m](X_i)$, where $X_i$'s are the connected components of $X$. It is enough to prove that $\mathcal{F}$ is a Nisnevich sheaf.
We need to show that for any elementary distinguished square in $Sm/k$
$$
 \xymatrix{                                   
W \ar[r] \ar[d] & V \ar[d]^p \\
U \ar[r]^i  & X,}
$$
the induced commutative square
$$
 \xymatrix{                                   
\mathcal{F}(X) \ar[r] \ar[d] & \mathcal{F}(V) \ar[d] \\
\mathcal{F}(U) \ar[r]  & \mathcal{F}(W)}
$$
is cartesian. By the construction of $\mathcal{F}$ we can assume that $X, W, V,U$ are connected. So, it is enough to prove that 
$$
 \xymatrix{                                   
\mathbb{Z}[\mathbb{G}_m](X) \ar[r] \ar[d] & \mathbb{Z}[\mathbb{G}_m](V) \ar[d] \\
\mathbb{Z}[\mathbb{G}_m](U) \ar[r]  & \mathbb{Z}[\mathbb{G}_m](W)}
 $$
is cartesian.
Let $a \in \mathbb{Z}[\mathbb{G}_m](U)$ and let $b \in \mathbb{Z}[\mathbb{G}_m](V)$ such that $a|_W = b|_W$. We can write $a = \sum_{i=1}^n a_i. f_i$ and $b = \sum_{j=1}^m b_j.g_j$, where $a_i, b_j \in \mathbb{Z} \setminus \left\{0\right\}$ and $(f_i, g_j) \in \mathbb{G}_m(U) \times \mathbb{G}_m(V)$ such that $f_i \neq f_{i'}$ and $g_j \neq g_{j'}$ for all $i \neq i'$ and $j \neq j'$. Since all the morphisms are dominant, $g_j|_W \neq  g_{j'}|_W$ and $f_i|_W \neq f_{i'}|_W$ for all $i \neq i'$ and $j \neq j'$. 
Hence, for every $i$ there exists atmost one $j$ such that $f_i|_W = g_j|_W$. 
Suppose for some $f_{i'}$, $f_{i'}|_W \neq g_j|_W$ for all $j$. Then we can write $$(\sum_{i=1}^n a_i. f_i|_W) - (\sum_{j=1}^m b_j.g_j|_W) = a_{i'}f_{i'} + \sum_{k =1}^l c_k.h_k = 0,$$ where $h_k \neq h_{k'}$ for all $k \neq k'$ and $f_{i'} \neq h_k$ for all $k$. This implies $a_{i'} = 0$, which gives a contradiction. Hence, for every $i$ there exists exactly one $j$ such that $f_i|_W = g_j|_W$. Therefore, $m =n$. Also we can write $a = \sum_{i=1}^n a'_i. f'_i$, such that $a'_i = b_i$ and $f'_i|_W = g_i|_W$. Since $\mathbb{G}_m$ is a Nisnevich sheaf, we get $g'_i \in \mathbb{G}_m(X)$ which restricts to $f'_i$ and $g_i$. This gives a section $c = \sum_{i =1}^n b_i.g'_i \in \mathbb{Z}[\mathbb{G}_m](X)$ which restricts to $a$ and $b$. The uniqueness of $c$ follows from the remark \ref{remark injective open}. 
\end{proof}

\bigskip

As $\mathbb{G}_m$ is pointed by $1$, $a_{Nis}(\mathbb{Z}[\mathbb{G}_m]) \cong \mathbb{Z} \oplus \mathbb{Z}(\mathbb{G}_m)$.  Here $\mathbb{Z}$ is the sheaf generated by the point $1$. 
Let $A$ be a sheaf of abelian groups on $Sm/k$. 
To give a morphism 
$\mathbb{G}_m \to A$, such that $1$ gets 
mapped to $0 \in A$, is equivalent to give a morphism 
$\mathbb{Z}(\mathbb{G}_m) \to A$ of abelian sheaves. 
Since $\mathbb{G}_m$ is $\mathbb{A}^1$-invariant, $a_{Nis}(\mathbb{Z}[\mathbb{G}_m])$ is $\mathbb{A}^1$-invariant. This implies $\mathbb{Z}(\mathbb{G}_m)$ is $\mathbb{A}^1$-invariant.

\begin{remark} \label{injective}
Let $\sigma_1 : \mathbb{G}_m \to \underline{K}_1^{MW}$ be the canonical pointed morphism (see \cite[page 86]{mor}). For any finitely generated separable field extension $F/k$, the morphism maps 
$u \in F^*$ to the corresponding symbol $[u] \in K_1^{MW}(F)$. The induced morphism $\mathbb{Z}(\mathbb{G}_m) \to  \underline{K}_1^{MW}$ is not injective. Indeed, we can choose $u \in F^* \setminus 1$ such that $u(u-1)$ is not $1$. The element $[u(u-1)] - [u] - [u-1] $ is zero in  $K_1^{MW}(F)$, but it is non zero in 
$\mathbb{Z}(\mathbb{G}_m)(F)$.
\end{remark}

\begin{lemma}
The $\mathbb{A}^1$-invariant sheaf of abelian groups $\mathbb{Z}(\mathbb{G}_m)$ is not strongly $\mathbb{A}^1$-invariant.
\end{lemma}
\begin{proof}
Suppose $\mathbb{Z}(\mathbb{G}_m)$ is strongly $\mathbb{A}^1$-invaraint. Then by 
\cite[Theorem 2.37]{mor}, the morphism $id : \mathbb{Z}(\mathbb{G}_m) \to \mathbb{Z}(\mathbb{G}_m)$ can be written as $\phi \circ \sigma_1$ for some unique 
$\phi$. This implies $\sigma_1$ is injective which contradicts remark \ref{injective}.
\end{proof}


\begin{thebibliography}{99}

\bibitem{bk} A.K. Bousfield, D.M. Kan. \emph{Homotopy limits, completions and localizations}. Lecture Notes in Mathematics, Vol. 304 Springer, Berlin, Heidelberg, New York (1972)

\bibitem{gi} P. Gille. \emph{Le probl\`eme de Kneser-Tits}. expos\'e Bourbaki n0 983,  Ast\'erisque 326  (2009), 39-81.

\bibitem{hir} P.S. Hirschhorn. \emph{Model Categories and their Localizations}. 
Mathematical Surveys and Monographs 99, American Mathematical Society, 2003. 

\bibitem{jar1} J.F. Jardine. \emph{Simplicial Presheaves}. J. Pure Applied Algebra 47 (1987), 35-87.

\bibitem{jar2} J.F. Jardine. \emph{Motivic symmetric spectra}. Doc. Math.5 (2000), 445-552.

\bibitem{thk} J.-L. Colliot-Th\'el\`ene, R.T. Hoobler, B. Kahn. \emph{The Bloch-Ogus-Gabber theorem}. Algebraic K-theory (Toronto, ON,  1996), 31–94, Fields Inst. Commun., 16, Amer. Math. Soc., Providence, RI, 1997. 

\bibitem{mor2} F. Morel. \emph{The stable $\mathbb{A}^1$-connectivity theorem}. K-theory, 2005, vol 35, pp 1-68.

\bibitem{mor} F, Morel. \emph{$\mathbb{A}^1$-algebraic topology over a field}. Lecture Notes in Mathematics 2052, Springer Verlag, to appear.

\bibitem{mor1} F, Morel. \emph{Friedlander-Milnor conjecturecof groups of small rank}. In Current Developments in Mathematics (2010). D. Jerison, B. Mazur, T. Mrowka, W. Schmid, R. Stanley, and S. T. Yau (editors).
International Press. 

\bibitem{mv} F. Morel, V. Voevodsky. \emph{$\mathbb{A}^1$-homotopy theory of schemes}. Publications Math\'ematiques de l'I.H.\'E.S, volume 90.

\bibitem{Wen} M, Wendt. \emph{$\mathbb{A}^1$-homotopy groups of Chevalley groups}. J. K-Theory 5 (2), 2010, pp. 245-287.



\end{thebibliography}
\end{document}